\documentclass[11pt]{article}
\usepackage{amsmath, amssymb, amsthm, mathrsfs, amscd, xypic}
\usepackage[dvips]{graphics} 
\usepackage{xypic}
\xyoption{all}
\newcommand{\ds}{\displaystyle}

\newcommand{\CC}{\mathbb{C}}
\newcommand{\NN}{\mathbb{N}} 
\newcommand{\FF}{\mathbb{F}}  
\newcommand{\KK}{\mathbb{K}}

\newcommand{\imag}{\operatorname{Im}}

\newcommand{\End}{\operatorname{End}}
\newcommand{\id}{\operatorname{id}}

\newcommand{\Hom}{\operatorname{Hom}}  
\newcommand{\RMod}{\operatorname{R-Mod}} 
\newcommand{\PMod}{\operatorname{\mathbb{K}[x]-Mod}}  
\newcommand{\msc}{\operatorname{WSP}} 
\newcommand{\msd}{\operatorname{LRep}(\Gamma)}

\newcommand{\Soc}{\operatorname{Soc}} 
\newcommand{\Ext}{\operatorname{Ext}}

\newcommand{\lh}{{\rm length}}

\newtheorem{theorem}{Theorem}
 
\newtheorem{corollary}{Corollary}

\newtheorem{question}{Question}

\title{On the Toeplitz-Jacobson algebra and direct finiteness}
\author{\sc Miodrag Cristian Iovanov \\ 
\vspace{.5cm} {\small  University of Iowa, USA and University of Bucharest, Romania }\\
\sc Alexander Sistko \\
{\small  University of Iowa, USA}
}

\begin{document}
\maketitle

\date{}

\begin{abstract}
We study the representation theory of the algebraic Toeplitz algebra $R=\KK\langle x,y\rangle/\langle xy-1\rangle$, give a few new structure and homological theorems, completely determine one-sided ideals and survey and re-obtain results from the existing literature as well. We discuss its connection to Kaplansky's direct finiteness conjecture, and a possible approach to it based on the module theory of $R$. In addition, we discuss the conjecture's connections to several central problems in mathematics, including Connes' embedding conjecture.
\end{abstract}


\section{Introduction}

The {\em{ Jacobson algebra}} or the {\em{ Toeplitz-Jacobson algebra}} is the associative $\KK$-algebra defined by $R = \KK\langle x,y \rangle /(xy-1)$, over some field $\KK$. Over time its relevance to contemporary mathematics has grown: it is a Leavitt path algebra, naturally connected to the direct finiteness of algebras (a universal algebra for this problem), and possesses interesting representation-theoretic and ring-theoretic properties \cite{K}. It can also be defined as the subalgebra of $\End_\KK(V)$ for a countable dimensional vector space $V=\KK^{(\aleph_0)}$, generated by the left and right shift operators defined with respect to a fixed basis, and analytic counterparts have been present in operator theory for a long time as well. Quite a bit is known about the Jacobson algebra; unfortunately, this information is spread across several sources. A partial aim of this article is to summarize and condense some of this information for easy reference, as well as provide a few new results. In general, we phrase these results in terms of the above presentation, which is of importance to the study of direct finiteness.

In particular, we classify completely the left, right and two-sided ideals, and directly recover known results on this algebra. We characterize finite-length modules in terms of Ext groups, determine Ext groups between finite length modules and explicitly construct an equivalence of categories in the vein of \cite{AB}, which produces a parameterization of finite length modules, in the form of a functor from finite dimensional representations of the Toeplitz graph to finite length modules over $R$. 

We aim to bring to the attention of a possibly diverse audience of mathematicians the connections between several long standing conjectures in mathematics. Hence, we briefly present and survey the current state of the Direct Finiteness Conjecture of Kaplansky, and recently discovered close connections with problems appearing from operator theory, group theory, symbolic dynamics, logic, and mathematical physics as well as to a central conjecture in mathematics, Connes' Embedding conjecture, and discuss possible new results from this perspective. Finally, we discuss possible new approaches to the Direct Finiteness Conjecture of Kaplansky, based on the representation theory of $R$.


Many of the representation-theoretic properties of $R$ can be attributed to its unique ideal structure. Indeed, we have the following, which the reader can confirm through direct computation: 

\begin{theorem} 
The socle of $R$ is the two-sided ideal $I = \langle 1-yx\rangle$, and it is the unique minimal two-sided ideal of $R$.  $I$ decomposes into a countable direct sum $\ds I = \bigoplus_{n=1}^{\infty}{S_n}$ of simple left modules (ideals), and similarly to the right, where $S_n$ is faithful, generated by the idempotent $f_n:= y^{n-1}x^{n-1}-y^nx^n$, and $S_n \cong S_1$ for $n \geq 1$.
\end{theorem} 

\begin{proof} 
See, for instance, \cite{AAJZ}, \cite{B}, \cite{Co}. 
\end{proof} 

Using the fact that $I$ is the unique minimal two-sided ideal of $R$, one sees that a simple module $T$ is either not faithful, in which case it is a simple module over the Laurent polynomials $R/I=\KK[X,X^{-1}]$, or it is faithful and any surjective morphism $\varphi:R\rightarrow T$ of left modules satisfies $I\not\subseteq \ker(\varphi)$. In this case $\varphi$ restricts to a surjection $\varphi\vert_I:I\rightarrow T$; hence, $T\cong S_1$. See, for instance, \cite{B}. The simple module $S=S_1$ can be viewed as the above mentioned representation of $R$ into $\End(V)$; indeed, $S=S_1={\rm Span}_\KK\{f_1,yf_1,y^2f_1,\dots,\}$ and $x$ and $y$ act as left and right shifts, respectively.


One can also recover this classification from general results on Leavitt path algebras. Indeed, $R$ is the Leavitt path algebra of the following graph, hereafter refered to simply as $\Gamma$:   

\begin{equation}
\xymatrix{{}_u\bullet & \bullet_v \ar[l]_{e}\ar@(dr,ur)[]_f}
\end{equation}

\bigskip

Theorem 1.1 of \cite{AR}  then applies, and we see that all simple $R$-modules are {\em{Chen modules}} \cite{Ch}.

The Jacobson algebra also enjoys nice homological properties. In particular, it is hereditary \cite{B}, \cite{AMP}. As a consequence, the classification of projective modules is reduced to the classification of left ideals; and every left ideal is isomorphic to either $S_1^{\oplus k}$ for some $k$, $I$, or $R$. Of course, this means that up to isomorphism, the only indecomposable projective is $S=S_1$. 

\section{The results}

\subsection*{Ideal structure}

In regards to ideals of $R$, we can be quite explicit; our aim is to provide a complete structural theorem (not only up to isomorphism) in terms of the generators $x,y$ (which are specific for this Leavitt path algebra), in view of the connections to direct finiteness. 

\begin{theorem} \label{s.1}
Every left ideal of $R$ is of the form $\Sigma \oplus Rp(x)$ where $p(x)$ is a polynomial and $\Sigma$ is contained in the socle of $R$.
\end{theorem} 

\begin{proof}
 Observe that $\ds R/Rx^k \cong \bigoplus_{i=1}^k{S_i}$ for each $k\geq 1$. It then follows that if an $R$-module $M$ possesses a generating set $G= \{ m_{\alpha} \mid \alpha \in A\}$, each of which is annihilated by a power of $x$, then $M$ is isomorphic to a direct sum of copies of $S_1$.  
Now, suppose that $M$ is a left ideal. Then the indexing set $A$ may be partitioned into $A = A_1 \cup A_2$, with $\alpha \in A_1$ if and only if $m_{\alpha}$ is annihilated by some power of $x$. But since $xy = 1$ and the collection $\{ y^ix^j \mid i,j \geq 0 \}$ forms a $\KK$-basis for $R$, for each $\alpha \in A_2$ there exists an $x^{k_{\alpha}}$ such that $p_{\alpha}(x) := x^{k_{\alpha}}m_{\alpha}$ is a polynomial in $x$. But then $\ds M/\left( \sum_{\alpha \in A_2}{Rp_{\alpha}(x)} \right)$ is generated by the cosets of elements in $G$, which are each annihilated by powers of $x$. Hence $\ds M/\left( \sum_{\alpha \in A_2}{Rp_{\alpha}(x)} \right)$ is isomorphic to a direct sum of copies of $S_1$, and in particular projective. 
In other words, $M$ may be written as  $M = \Sigma \oplus P$, where $\Sigma$ is contained in the socle of $R$ and $P$ is generated by the polynomials $\{ p_{\alpha}(x)\mid \alpha \in A_2 \}$. Since $P$ necessarily contains the ideal of $\KK[x]$ generated by the $p_{\alpha}$'s, there exists a single $p(x) \in \KK[x]$ generating this left ideal $P$.
\end{proof}  

\begin{corollary} 
Let $\Sigma \oplus Rp(x)$ be as in the previous theorem. If $p(x) \neq 0$ then $\Sigma \oplus Rp(x)$ is finitely-generated (and $\Sigma$ is of finite length). In particular, every left ideal of $R$ is either semisimple or finitely-generated.
\end{corollary}  

\begin{proof} 
The $R$-module homomorphism $R \rightarrow Rp(x)$ mapping $1\mapsto p(x)$ is an isomorphism, and therefore maps the socle of $R$ onto the socle of $Rp(x)$. In other words, $Ip(x) = \Soc (Rp(x))$. Since $Rp(x)$ is a left ideal, we must also have $Ip(x) = Rp(x) \cap I$. Now, the injection $I \rightarrow R$ descends to an injection $I/Ip(x) = I/(I \cap Rp(x)) \rightarrow R/Rp(x)$, and the image of this map is $(I+Rp(x))/Rp(x)$, and $(R/Rp(x))/(I+Rp(x))/Rp(x) \cong R/(Rp(x)+I) \cong \KK[x,x^{-1}]/(p)$. Hence, we have a short exact sequence $0 \rightarrow I/Ip(x) \rightarrow R/Rp(x) \rightarrow \KK[x,x^{-1}]/(p) \rightarrow 0$. Now, $I/Ip(x)$ is a quotient of a semisimple module, and is therefore semisimple.  We claim that $\ds \bigoplus_{i=1}^d{S_i} \oplus Ip(x) = I$, where $d=\deg(p)$, $p=\sum\limits_{i=0}^d\alpha_ix^i$. Indeed, the elements $f_1,\ldots , f_d$ are clearly in the $R$-module $\ds \bigoplus_{i=1}^d{S_i} + Ip(x)$.  But then $\ds f_1p(x) = f_1\sum_{i=0}^d{\alpha_ix^i} = \sum_{i=0}^d{\alpha_if_1x^i} = \sum_{i=0}^d{\alpha_ix^if_{i+1}} = x^df_{d+1} + \sum_{i=0}^{d-1}{\alpha_ix^if_{i+1}}$, 
which implies that $\ds x^df_{d+1} \in \bigoplus_{i=1}^d{S_i} + Ip(x)$.  Since $S_{d+1}$ is simple, this means $\ds S_{d+1} \subseteq \bigoplus_{i=1}^d{S_i} + Ip(x)$.  Repeating this with the elements $f_kp(x)$ for $k > 1$, we can prove by induction that $\ds S_{d + k} \subseteq \bigoplus_{i=1}^d{S_i} + Ip(x)$.  In other words, $\ds I = \bigoplus_{i=1}^d{S_i} + Ip(x)$.  For $\ds u = \sum_{i=1}^n{a_if_i} \in I$ with $a_nf_n \neq 0$ ($a_i\in R$), we compute $\ds \sum{a_if_ip(x)} = \sum_{k=1}^n{\sum_{i=0}^d{\alpha_ia_kx^if_{k+i}}} = a_nx^df_{d+n} + \sum_{k+i<d+n}{\alpha_ia_kx^if_{k+i}}$.  But for $a_nf_n \neq 0$, we may assume that $a_n$ is a nonzero element in the span of the set $\{ y^ix^j \mid j = 0,\ldots , n-1 \}$.  Hence, $a_nx^d$ cannot annihilate $f_{d+n}$, so that $a_nx^df_{d+n} \neq 0$. Therefore, $up(x)\not\in \bigoplus\limits_{i=1}^dS_i$; hence $\ds  Ip(x) \cap \bigoplus_{i=1}^d{S_i} = \{ 0 \}$, and $\ds \bigoplus_{i=1}^d{S_i} \oplus Ip(x) = I$.  This implies $\ds I/Ip(x) = \bigoplus_{i=1}^d{S_i}$, so that $R/Rp(x)$ fits into a short exact sequence $\ds 0 \rightarrow \bigoplus_{i=1}^{\deg p}{S_1}\rightarrow R/Rp \rightarrow \KK[x,x^{-1}]/(p) \rightarrow 0$. In particular, $R/Rp$ has finite length. Therefore, so does $(\Sigma \oplus Rp)/Rp\cong \Sigma$. Hence, $\Sigma\oplus Rp$ is finitely generated.
\end{proof} 

In particular, this recovers the classification of projectives in a direct way. Note that since $1=f_1+yx$, $R\cong R\oplus S$ and so $R\cong R\oplus \Sigma$ for every finitely generated semisimple module $\Sigma$ (this relation is a particular instance of the Ara-Mareno-Pardo Realization Theorem \cite{Ab1}.) In particular, since $Rp(x)\cong R$ and $S$ is projective, by the previous Corollary we see directly that $R$ is hereditary. Therefore, every projective is a direct sum of left ideals, and so, we get

\begin{corollary} 
Up to isomorphism, every projective $R$-module $P$ is of the form $P=R^{(\alpha)}\oplus S^{(\beta)}$ for cardinalities $\alpha,\beta$. Moreover, $(\alpha,\beta)$ can be chosen in one of the following two forms, which completely describe the isomorphism type of $P$:
\begin{center}
(1) $\alpha$ finite, $\beta=0$ or $\beta\geq\aleph_0$\\
(2) $\alpha$ infinite, $\beta=0$ or $\beta>\alpha$.
\end{center}
\end{corollary} 
\begin{proof}
This follows since $R^\alpha\oplus S^\beta\cong R^\alpha$ whenever $\beta\leq \alpha$ or they are both finite.
\end{proof}


We now note that for an ideal $H=\Sigma\oplus Rp(x)$ as above, then of course, $p(x)$ is not unique: we have $R=\bigoplus\limits_{i=1}^dS_i\oplus Rx^d$, and hence 
\begin{equation}\label{e2}
H=\Sigma\oplus Rp(x)=\Sigma\oplus (\bigoplus\limits_{i=1}^dS_i\oplus Rx^d)p(x)=\left(\Sigma\oplus(\bigoplus\limits_{i=1}^dS_i)p(x)\right)\oplus Rx^dp(x).
\end{equation} 
Nevertheless, such a $p$ of minimal degree is, in fact, unique up to scalar multiplication. Let $H\cap \KK[x]$ be generated by $h(x)$ as an ideal of $\KK[x]$ (which is isomorphic to polynomials in one variable over $\KK$); that is, $h(x)$  is a polynomial of minimal degree in $H$. If $p(x)$ is as in (\ref{e2}), then $p(x)$ generates $H+I/I$ in $R/I=\KK[x,x^{-1}]$, so $h(x)=\lambda x^dp(x)$ or $p(x)=\lambda x^dh(x)$ modulo $I$ for some $d\geq 0$ and $\lambda\in\KK^\times$. Since every element of $I$ is annihilated by a power of $x$, we get either $x^nh(x)=\lambda x^{n+d}p(x)$ or $x^np(x)=\lambda x^{n+d}h(x)$ for some $n$. By canceling inside $\KK[x]$ and  using the minimality of $h$, it must be that $p(x)=\lambda x^dh(x)$. But, again by equation (2), we now see that $h(x)$ has the same property: $H=\Sigma'\oplus Rh(x)$ for some semisimple $\Sigma'\subseteq I$. Hence, $p(x)$ will be uniquely determined up to scalar multiplication by either of the conditions that (1) it is of minimal degree in $H$ or (2) it is of minimal degree for which $H=\Sigma\oplus Rp(x)$ for some semisimple $\Sigma$. 

Also, note that if $H=\Sigma\oplus Rp(x)$, then $\lh(\Sigma)\leq \deg(p)$, since $R/Rp\cong S^{\deg(p)}$ as seen before. Furthermore, since $R=Rp(x)\oplus \bigoplus\limits_{i=1}^dS_i$ so $R/Rp(x)\cong \bigoplus\limits_{i=1}^dS_i$, given $H\supseteq Rp(x)$, there is a unique left submodule  $\Sigma(H)$ of $\bigoplus\limits_{i=1}^d{S_i}$ for which $H=\Sigma(H)\oplus Rp(x)$.  Hence, we may summarize the structure theorem as follows:

\begin{theorem}
If $H$ is a left ideal of $R$, then either $H$ is semisimple (so contained in $I$), or there is a unique monic polynomial $p(x)\in H$ of minimal degree, and $H$ is of the form $H=\Sigma(H)\oplus Rp(x)$, for a uniquely determined submodule $\Sigma(H)\subseteq S_1\oplus\dots\oplus S_{\deg(p)}$. 
\end{theorem}

We further remark that the above submodule $\Sigma$ of $I$ is uniquely determined as follows. Note that $S_1\cong S_n$ by an isomorphism taking $f_1$ to $x^{n-1}f_n$ (and the basis $\{f_1,yf_1,y^2f_1,\dots\}$ to the corresponding basis $\{x^{n-1}f_n,yx^{n-1}f_n,y^2x^{n-1}f_n\dots\}$). The span of $I_0=\{x^{k-1}f_k|k\geq 1\}$ is the socle of $I$ regarded as a $\KK[x]$-module; in fact, $I$ is a direct sum of countably many copies of the injective hull of $\KK[x]/(x)$, regarded as a $\KK[x]$-module, and it is injective over $\KK[x]$ (as $\KK[x]$ is Noetherian; we will use this fact later). Then $\Sigma_0={\rm Span}_\KK(I_0)\cap \Sigma$ is the $\KK[x]$-socle of $\Sigma$, and this completely determines $\Sigma$ via the action of $y$: $\Sigma=R\Sigma_0=\KK[y]\Sigma_0$. This gives a complete set of parameters completely and uniquely determining any given non-semisimple ideal of $H$, in the form of a polynomial $p(x)$ and a finite dimensional subspace $L$ of $I_0$, contained in ${\rm Span}\{f_1,xf_2,\dots,x^{d-1}f_d\}$, $d=\deg(p)$. When $H$ is semisimple and contained in $I$, then again, its $\KK[x]$-socle determines it as above. Hence we have
\begin{corollary}
Any ideal of $R$ can be written as a direct sum of $R$-submodules
\begin{equation}
H=\KK[y]L\oplus Rp(x)
\end{equation}
where if $H$ is non-semisimple then $p(x)$ the unique non-zero polynomial in $x$ belonging to $H$, and $L=H\cap {\rm Span}_\KK(I_0)\cap S_1\oplus\dots S_{\deg(p)}$; if $H$ is semisimple, $p(x)=0$, $L=H\cap {\rm Span}_\KK I_0$ and $H=\KK[y]L$. 
\end{corollary}

\subsection*{Determining Ext spaces}

The above gives us information on cyclic $R$-modules, which are isomorphic to quotients of $R$ by left ideals. If, in particular, $R/L$ is a cyclic module with $L$ finitely-generated non-semisimple, the above proof can be modified to show that $R/L$ has finite length. For general finite-length modules we have the following: 

\begin{theorem} 
Every finite-length module $M$ is the middle term of a short exact sequence $\ds 0 \rightarrow S_1^{\oplus k} \rightarrow M \rightarrow F \rightarrow 0$, where $k \in \NN$ and $F$ is finite-dimensional. All finite-dimensional $R$-modules are direct sums of the modules $L_p = \KK[x,x^{-1}]/(p)$, $p$ a (not-necessarily irreducible) polynomial, and there is a natural identification $\ds \Ext^1(L_p, S_1^{\oplus k}) \cong \bigoplus_{i=1}^k{\KK[T]/(p^*(T))}=\left(\KK[T]/(p^*(T))\right)^{\oplus k}$, where $p^*$ is the polynomial defined by $p^*(y) = p(x)y^{\deg (p)} \in \KK[y] \subseteq R$.
\end{theorem}

\begin{proof} 
The first claim comes from the classification of simple $R$-modules, and the fact that $S_1$ is projective. The case where $L_p$ is a simple module in the sense of \cite{Ch} is discussed in \cite{Ab1}. In general, we start by verifying that $Rp^* = Rp + I$.  Indeed, $Rp + I$ is a left ideal whose correspondent in the commutative quotient algebra $R/I$ is two-sided, hence is itself a two-sided ideal. Therefore, $p^* = py^n \in Rp + I$, so $Rp^* \subset Rp+I$.  If $\ds p(x) = \sum_{i=0}^n{\alpha_ix^i}$, then $p^*(y) = \sum_{i=0}^n{\alpha_iy^{n-i}}$.  Therefore, $\ds x^np^*(y) = \sum_{i=0}^n{\alpha_ix^ny^{n-i}} = \sum_{i=0}^n{\alpha_i x^i} = p(x)$, so that $Rp \subset Rp^*$.  Therefore, since $\ds I = \bigoplus_{i=1}^n{S_i} \oplus Ip$, to prove the reverse containment it suffices to show that $f_k \in Rp^*$ for all $k \le n$.  We first compute $\ds f_1p^* = (1-yx)\sum_{i=0}^n{\alpha_iy^{n-i}} = \sum_{i=0}^n{\alpha_iy^{n-i}} - \alpha_nyx - \sum_{i=0}^{n-1}{\alpha_iy^{n-i}} = \alpha_n - \alpha_nyx = \alpha_nf_1$.  Since $\alpha_n \neq 0$, it follows that $f_1 \in Rp^*$.  Now, assume that $1 \le k \le n$, and that $f_i \in Rp^*$ for all $1 \le i < k$.  Then $\ds f_kp^* = \sum_{i=0}^n{\alpha_if_ky^{n-i}} = \sum_{i=0}^n{\alpha_iy^{n-i}f_{k-(n-i)}} = \alpha_nf_k + \sum_{i=0}^{n-1}{\alpha_iy^{n-i}f_{k-(n-i)}}$, where $f_j = 0$ for $j < 1$.  By the induction hypothesis, we conclude that $f_k \in Rp^*$, which proves the desired equality. 

It follows that $R/Rp^* \cong L_p$, so we obtain a short exact sequence $0 \rightarrow Rp^* \rightarrow R \rightarrow L_p \rightarrow 0$.  Applying $\ds \Hom (- , S_1^{\oplus k})$ to this sequence yields an exact sequence 

\begin{center} 
$\ds 0 \rightarrow \Hom (L_p , S_1^{\oplus k}) \rightarrow \Hom (R, S_1^{\oplus k}) \rightarrow \Hom (Rp^*, S_1^{\oplus k} ) \rightarrow \Ext^1 (L_p , S_1^{\oplus k}) \rightarrow \Ext^1(R, S_1^{\oplus k} ) = 0$
\end{center} 
where the last equality follows since $R$ is projective. Also note that $\ds \Hom (L_p, S_1^{\oplus k} ) = 0$, since $S_1$ is simple and infinite dimensional, while $L_p$ is always finite-dimensional.  To compute $\ds \Ext^1(L_p, S_1^{\oplus k} )$, take a non-zero morphism $\phi : Rp^* \rightarrow S_1^{\oplus k}$.  Consider the problem of extending $\phi$ to a morphism $\bar{\phi } : R \rightarrow S_1^{\oplus k}$. 
We identify $S_1$ with $R/Rx$, with basis $\{ \overline{y^i} \mid i \geq 0 \}$.  If $\bar{\phi}$ extends $\phi$, then $\ds p^*\bar{\phi}(1) = \bar{\phi}(p^*) = \phi (p^*)$.  But there exist points ${\bf{f}} = (f_1,\ldots , f_k), {\bf{g}} = (g_1,\ldots , g_k) \in \KK[y]^k$ such that $\phi (p^*) = \overline{{\bf{f}}}$ and $\bar{\phi}(1) = \overline{{\bf{g}}}$, where $\overline{\bf{f}}=(\overline{f_1},\dots,\overline{f_k})$.  It follows that $\overline{{\bf{f}}} = p^*\overline{{\bf{g}}} = \overline{p^*{\bf{g}}}$.  Since $Rx$ contains no non-zero polynomials in $y$, it follows that $f_i(y) = p^*g_i(y)$ in $R$, for all $i$.  Conversely, suppose that ${\bf{f}}$ and ${\bf{g}}$ are given, satisfying this equation.  Then since $Rp^*$ and $R$ are free, there exist unique $R$-module morphisms $\phi$, $\bar{\phi}$ satisfying $\phi (p^*) = \overline{{\bf{f}}}$ and $\bar{\phi }(1) = \overline{{\bf{g}}}$.  Then $\bar{\phi}$ extends $\phi$, i.e. the above diagram commutes.  In other words, the map $\ds \Theta : \Hom (Rp^* , S_1^{\oplus k}) \rightarrow \bigoplus_{i=1}^k{\KK[T]/(p^*(T))}$ defined by $\Theta (\phi ) = {\bf{f}}(T)$ is surjective, with kernel $\Hom (R,S_1^{\oplus k})$.  Hence $\ds \Ext^1(L_{p}, S_1^{\oplus k}) \cong \bigoplus_{i=1}^k{\KK[T]/(p^*(T))}$ naturally, as we wished to show.
\end{proof}

Of course, we could have just done the computation above for $k=1$, but we wanted to emphasize how the extensions are obtained in general. In fact, we may use this result to compute the dimension of $\Ext^1(M_1,M_2)$, for $M_1$ and $M_2$ finite-length modules. To see this, again recall that if $M$ has finite-length, then $M$ is the middle term of a short exact sequence $0 \rightarrow S_1^{\oplus k} \rightarrow M \rightarrow F \rightarrow 0$, where $F$ is finite-dimensional. In fact, $S_1^{\oplus k } = IM$, where $I = \Soc (R)$ and $k$ is the number of times that $S_1$ appears as a factor in a composition series for $M$. Of course, $F = M/IM$. It will be useful to define $d(M)$ to be the largest non-negative integer for which $\ds M = S_1^{\oplus d(M)} \oplus M'$ for some submodule $M'$; this is obviously well defined by Krull-Schmidt.  

Decomposing $F$ into indecomposable components, we are able to compute $\ds \Ext^1(F,S_1)$ and $\Ext^1(F, S_1^{\oplus k} )$ using the above theorem and the finite additivity properties of $\Ext^1(\cdot , \cdot )$. Furthermore, since $S_1$ is projective, $\Ext(M, S_1^{\oplus k} ) = \Ext^1(S_1^{\oplus d(M)}\oplus M' , S_1^k ) = \Ext^1(S_1^{\oplus d(M)} , S_1^{\oplus k }) \oplus \Ext^1(M', S_1^{\oplus k}) = \Ext^1(M', S_1^{\oplus k})$. So to compute $\ds \Ext^1(M, S_1^{\oplus k})$, we may assume without loss of generality that $d(M) = 0$, i.e. that $M$ has no $S_1$-direct summands; the general case reduces to this by working with $M/S_1^{d(M)}$. Then we have the following: 

\begin{corollary} 
If $M$ has finite length and no $S_1$-direct summands, then $\ds \dim_{\KK}\Ext^1(M, S_1^{\oplus k}) = k\left[ \dim_{\KK}M/IM - \ell (IM) \right]$, where $\ell (IM)$ denotes the length of $IM$. In the general case when $d(M)\neq 0$, $\ds \dim_{\KK}\Ext^1(M, S_1^{\oplus k}) =\dim_{\KK}\Ext^1(M/S_1^{d(M)}, S_1^{\oplus k})$.
\end{corollary} 

\begin{proof} 
Apply $\Hom_R(-,S_1)$ to the short exact sequence $0 \rightarrow IM \rightarrow M \rightarrow M/IM \rightarrow 0$ to get the long exact sequence $0 \rightarrow \Hom_R(M/IM, S_1) \rightarrow \Hom_R(M,S_1) \rightarrow \Hom_R(IM, S_1) \rightarrow \Ext^1(M/IM, S_1) \rightarrow \Ext^1(M, S_1) \rightarrow \Ext^1(IM, S_1) = 0$, since $IM$ is projective. We claim that $\Hom_R(M,S_1) = 0$. Indeed, if there is a $0 \neq \varphi \in \Hom_R(M, S_1)$, then $\varphi$ is surjective and hence $M/\ker \varphi = S_1$. Since $S_1$ is projective, $M = \ker \varphi \oplus S_1$, contrary to hypothesis. So we have a short exact sequence $0 \rightarrow \Hom_R(IM, S_1)\rightarrow \Ext^1(M/IM, S_1) \rightarrow \Ext^1(M,S_1) \rightarrow 0$, and hence $\ds \dim_{\KK}\Ext^1(M,S_1) = \dim_{\KK}\Ext^1(M/IM,S_1) - \dim_{\KK} \Hom_R(IM, S_1)$. But $M/IM$ is finite-dimensional, and it easily follows from the previous theorem that $\dim_{\KK}\Ext^1(M/IM, S_1) = \dim_{\KK}M/IM$. Furthermore, $\ds \dim_{\KK}\Hom_R(IM,S_1) = \ell (IM)$ since $\Hom_R(S_1,S_1) = \KK$, and the general formula follows by finite additivity of $\Ext^1(M, - )$.
\end{proof} 

For an $R$-module $M$, we will use $\operatorname{lf}(M)$ to denote the sum of all finite-dimensional submodules of $M$. It is the largest locally finite submodule of $M$. We can then compute the dimensions of Ext groups between finite-length modules. More explicitly, we reduce the computations (recursively) to computations of Ext- and Hom- groups between finite-dimensional $\KK[x,x^{-1}]$-modules, where the answers are already known by classical PID theory; in the most general case, the answer is written also in terms of the dimension of the Hom-space $\Hom(M,N)$. Note that if $M,N$ are of finite length, then $\Hom_R(M,N)$ is finite dimensional, which follows inductively in a standard argument because $\dim(\End_R(T))<\infty$ for every simple $R$-module $T$.

\begin{corollary} 
Suppose $M$ and $N$ are finite-length $R$-modules 
and that $F$ is a finite-dimensional $R$-module. Then the following formulas hold: 
\begin{eqnarray*}
(i)\,\,\,\,\,\, \dim_{\KK}\Ext^1(F,N) & = & \dim_{\KK}\Ext^1_{\KK[x,x^{-1}]}(F, N/IN) + \ell (IN)\dim_{\KK}F\\ & & + \dim_{\KK}\Hom_{\KK[x,x^{-1}]}(F, \operatorname{lf}(N)) - \dim_{\KK}\Hom_{\KK[x,x^{-1}]}(F, N/IN). \\
(ii)\,\,\,\,\,\, \dim_{\KK}\Ext^1(M,N) & = & \dim_{\KK}\Ext^1(M/IM,N) +  \dim_{\KK}\Hom_R(M,N) - \\
 & & \dim\Hom_{\KK[x,x^{-1}]}(M/IM,\operatorname{lf}(N))-\ell(IM)\ell(IN). \\
(iii)\,\,\,\,\,\, \dim_\KK\Ext^1(M,F) & = & \dim_\KK \Ext^1_{\KK[x,x^{-1}]}(M/IM,F).\\
(iv)\,\,\,\,\,\, \dim_{\KK}\Ext^1(M,N) & = & \dim_\KK\Ext^1_{\KK[x,x^{-1}]}(M/IM,N/IN)
+\dim_\KK(\Ext^1(M,IN)) \\ 
& & +\dim_\KK\Hom_R(M,N)-\dim_\KK\Hom_{\KK[x,x^{-1}]}(M/IM,N/IN)\\
& & -d(M)\ell(IN).
\end{eqnarray*} 
\end{corollary} 

\begin{proof} 
For the first formula, apply $\Hom_R(F, -)$ to the sequence $0 \rightarrow IN \rightarrow N \rightarrow N/IN \rightarrow 0$ to get the long exact sequence $0 \rightarrow \Hom_R(F,IN) \rightarrow \Hom_R(F,N) \rightarrow \Hom_R(F,N/IN) \rightarrow \Ext^1(F,IN) \rightarrow \Ext^1(F,N) \rightarrow \Ext^1(F,N/IN) \rightarrow 0$, where the higher-order terms vanish because $R$ is hereditary. It is then straightforward to verify that $\Hom_R(F, N) = \Hom_{\KK[x,x^{-1}]}(F, \operatorname{lf}(N))$ and that the map $\ds \Hom_R(F,N) \rightarrow \Hom_R(F,N/IN)$ is injective ($\Hom_R(F,IN)=0$). In other words, we have a short exact sequence $0 \rightarrow \Hom_{\KK[x,x^{-1}]}(F, \operatorname{lf}(N)) \rightarrow \Hom_{\KK[x,x^{-1}]}(F, N/IN) \rightarrow \Ext^1(F, IN) \rightarrow \Ext^1(F,N) \rightarrow \Ext^1_{\KK[x,x^{-1}]}(F,N/IN) \rightarrow 0$, which in combination with Theorem 4 implies the first formula. 

For the second, apply $\Hom_R(-,N)$ to the short exact sequence $0 \rightarrow IM \rightarrow M \rightarrow M/IM \rightarrow 0$ to get the exact sequence

\begin{center}
 $0\rightarrow\Hom (M/IM,N)\rightarrow\Hom (M,N)\rightarrow\Hom (IM,N)\rightarrow\Ext^1(M/IM,N)\rightarrow \Ext^1(M,N)\rightarrow 0$
\end{center} 

Since $\Hom_R(M/IM,N)=\Hom_{\KK[x,x^{-1}]}(M/IM,\operatorname{lf}(N))$, $\dim_\KK(\Hom(M,N))<\infty$ and furthermore $\Hom_R(IM,N)=\Hom_R(IM,IN)=\ell(IM)\ell(IN)$, this yields the required result. \\
(iii) and (iv) are similar.
\end{proof}

\subsection*{A categorical equivalence and parametrizations of $R$-modules}

 Obtaining information on general $R$-modules is not as easy, but there are still things we can say. Somewhat surprisingly, we can derive a connection between the representation theory of $R$ and that of $\Gamma$. We provide an explicit construction below.

To start, pick an $R$-module $M$ and consider the canonical short exact sequence $0 \rightarrow IM \rightarrow M \rightarrow M/IM \rightarrow 0$. As previously mentioned, $IM$ is the sum of all faithful simple submodules of $M$ (the $S$-socle) and $M/IM$ is annihilated by $I$, hence a module over $\KK[x,x^{-1}]$. Furthermore, $S_1$ is the injective hull of the $\KK[x]$-module $\KK[x]/(x)$ (as before, regarded via restriction $\KK[x]\subset R$). Since $\KK[x]$ is Noetherian, $IM$ is an injective $\KK[x]$-module. Pick a splitting homomorphism $\alpha : M/IM \rightarrow M$ for this short exact sequence (considered now in $\PMod$.) We will call the pair $(M, \alpha )$ a {\em{weak splitting pair}}. We define a category $\msc$  whose objects are weak splitting pairs, and whose morphisms $\varphi : (M, \alpha  ) \rightarrow (N,\beta )$ are $R$-module homomorphisms $\varphi : M \rightarrow N$ such that $\imag (\varphi \circ \alpha ) \subset \imag \beta$. It is straightforward to note that this condition is equivalent to (the more natural condition) $\beta\circ \varphi=\varphi\circ \alpha$.

Given a weak splitting pair $\pi = (M,\alpha )$, we have $M = IM \oplus \imag \alpha$ as $\KK[x]$-modules. Note that left multiplication by $y$ leaves $IM$ invariant, but for each $m \in \imag \alpha$, we must write $ym = m_1 +m_2$, where $m_1 \in IM$ and $m_2 \in \imag \alpha$. Then $m \in \imag \alpha$ and $m = (xy)m = x(ym) = x(m_1+m_2) = xm_1 + xm_2$ implies that $xm_1 = 0$ (and $xm_2 = x$). We may then define a $\KK$-linear map $\psi_e(\pi ):\imag \alpha \rightarrow M_0$ as $\psi_e(\pi)(m) = m_1$, where $N_0$ is defined for an arbitrary $R$-module $N$ as the $\KK[x]$-socle of $IN$, i.e. $N_0 = \{ n \in N \mid xn = 0 \}$. If $N\subseteq R$, then $N_0=N\cap I_0$ in the notation from before. Let $\psi_f=\psi_f(\pi ) : \imag \alpha \rightarrow \imag \alpha$ be the left multiplication by $x$, and let $M_u(\pi) = M_0$ and $M_v(\pi ) = \imag \alpha$. The data $\Xi (\pi) =  (M_u, M_v, \psi_e, \psi_f )$ defines a representation of $\Gamma$ in the notations from Section 1. 

\begin{theorem} 
 $(M,\alpha ) \rightarrow \Xi (M,\alpha )$ induces a functor $\Xi : \msc \rightarrow \operatorname{Rep}(\Gamma )$. After corestriction, it is a functor $\Xi : \msc \rightarrow \msd$, where $\msd$ is the full subcategory of $\operatorname{Rep}(\Gamma)$ consisting of representations with an invertible map on the loop $f$. 
\end{theorem} 

\begin{proof} Given two weak splitting pairs $\pi_1 = (M,\alpha )$ and $\pi_2 = (M,\beta )$ and a morphism $\varphi : (M,\alpha ) \rightarrow (N,\beta )$ in $\msc$, we define $\Xi(\varphi )$ as follows: define $\Xi(\varphi)_u : M_0 \rightarrow N_0$ and $\Xi(\varphi)_v : \imag \alpha \rightarrow \imag \beta$ to be the restrictions of $\varphi$ to their respective domains. Since $\varphi$ is a morphism in $\msc$, these maps are well-defined. That $\Xi(\varphi )$ is a morphism $\Xi (\pi_1) \rightarrow \Xi (\pi_2)$ is equivalent to $\Xi(\varphi)_u \circ \psi_e(\pi_1) = \psi_e(\pi_2) \circ \Xi(\varphi )_v$ and $ \Xi(\varphi)_v\circ\psi_f(\pi_1) = \psi_f(\pi_2) \circ \Xi(\varphi )_v$. The latter equality is trivial, as $\Xi(\varphi)_v$ is $\KK[x]$-linear. The former reduces to showing  that $(\varphi (m))_1 = \varphi (m_1)$ for all $m \in \imag \alpha$, in the notation defining $\psi_e$ above. This is simple, since $\varphi (m)_1 + \varphi(m)_2 = y\varphi (m) = \varphi (ym) = \varphi (m_1) + \varphi (m_2)$ and $\varphi (m_1) \in N_0$, $\varphi (m_2) \in \imag \beta$. Finally we note that $\Xi$ preserves compositions $\pi_1 \xrightarrow[]{\varphi_1} \pi_2 \xrightarrow[]{\varphi_2} \pi_3$ in $\msc$ because the maps at $u$ and $v$ are simply restrictions of the $\varphi_i$'s. For the second statement, simply note that $\imag \alpha$ is a $\KK[x]$-module isomorphic to $M/IM$. But since left multiplication by $x$ is invertible on $M/IM$,  $\psi_f(\pi )$ must be invertible as well.
\end{proof}

 If we start with an arbitrary $\rho = (M_u,M_v, \psi_e, \psi_f )$ in $\msd$, we may construct a weak splitting pair $(M, \alpha )$ as follows: as a vector space, set $M = (S_1\otimes_{\KK}M_u) \oplus M_v$. Let $x$ act on $M$ as $x((rf_1)\otimes m_u, m_v) = ( (xrf_1)\otimes m_u, \psi_f(m_v) )$, and $y$ as $y((rf_1)\otimes m_u, m_v) = ((yrf_1)\otimes m_u + f_1\otimes\psi_e(m_v), \psi_f^{-1}(m_v))$. It is then easy to check that for any $m \in M$, $m = x(ym)$. We set $\alpha : M_v \rightarrow M$ to be the inclusion map. A morphism $\varphi : (M_u,M_v,\psi_e,\psi_f) \rightarrow (M_u',M_v', \psi_e', \psi_f' )$ in $\msd$ induces an $R$-module morphism $\Psi(\varphi)= (\id_{S_1}\otimes \varphi_u)\oplus \varphi_v : M(\rho ) \rightarrow M(\rho ')$ which satisfies $\Psi(\varphi)(M_v) \subset M_v'$, and under this construction morphisms in $\msd$ are preserved. In other words:  

\begin{theorem}
The above construction induces a functor $\Psi : \msd \rightarrow \msc$. The functors $\Xi$ and $\Psi$ induce an equivalence of categories $\msc \cong \msd$.
\end{theorem}  

\begin{proof} 
Clearly $\Psi (\id_{\rho} ) = \id_{\Psi(\rho )}$ for any representation $\rho$, and $\ds \Psi (\varphi \circ \varphi ') = (\id_{S_1}\otimes [\varphi\circ \varphi']_u)\oplus ([\varphi \circ \varphi']_v) = \left( (\id_{S_1}\otimes \varphi_u)\oplus \varphi_v \right) \circ \left((\id_{S_1}\otimes \varphi'_u)\oplus \varphi'_v \right)= \Psi (\varphi ) \circ \Psi (\varphi' )$, so $\Psi$ is a functor. For the second, we note that if $\pi = (M,\alpha )$ is given, then $\Psi(\Xi(\pi )) = (S_1\otimes M_0)\oplus \imag \alpha \cong IM \oplus \imag \alpha$. Furthermore, $x(\sigma , a) = (x\sigma , \psi_f(\pi)a) = (x\sigma, xa )$ and $y(\sigma , a) = (y\sigma + \psi_e(\pi)a , \psi_f(\pi)^{-1}a) = (y\sigma + a_1 , \psi_f(\pi)^{-1}a) = (y\sigma + a_1, a_2)$, since $a = xa_2$ and $x$ is invertible in $\imag \alpha$. Clearly then, $\Psi(\Xi(\pi)) \cong \pi$. The proof that $\Xi(\Psi(\rho)) \cong \rho$ for all $\rho \in \msd$ works similarly. So $\Psi$ and $\Xi$ are inverses at the level of objects. At the level of morphisms, $[\Psi\circ \Xi](\varphi )$ maps any $m \in IM$ to $\varphi (m)$, and any $m \in \imag \alpha$ to $\varphi (m)$. But by the above argument, $M \cong (S_1\otimes M_0) \oplus \imag \alpha$ and $N \cong (S_1\otimes N_0) \oplus \imag \beta$ naturally, so that $[\Psi \circ \Xi ](\varphi ) = \varphi$ with the natural identifications. By a similar argument, $[\Xi \circ \Psi](\varphi ) = \varphi$. 
\end{proof} 

Since there is a natural forgetful functor $U : \msc \rightarrow \RMod$ defined by $U((M,\alpha))=M$ which is surjective on objects, one can see that we may realize $\RMod$ as a quotient category of $\msd$ via the composition $U \circ \Psi : \msd \rightarrow \RMod$. Also, through the functor $\Xi$, it is easy to see that objects $(M,\alpha)$ with $M$ of finite length over $R$ correspond to finite dimensional $\Gamma$-representations. Moreover, the category $\msd$ is a category of modules: representations for which $f$ acts as an invertible element can be viewed as modules over the path algebra $\KK[\Gamma]$ with a relation making $f$ invertible, i.e. $\KK[\Gamma]\langle f^{-1}\rangle=\KK[\Gamma]*\KK[W]/\langle fW-1,Wf-1\rangle$ (of course, this is not an admissible ideal in the sense of representation theory of finite dimensional algebras, but we get an algebra nevertheless).  Results of similar flavor are given in \cite{AB} for  arbitrary Leavitt path algebras; our emphasis is again on the specific generators $x,y$. Although $\KK\Gamma$ is known to be of wild representation type, it is our hope that the above equivalence could be exploited to find new invariants for modules over the Jacobson algebra. It also opens up the possibility of studying $R$-modules geometrically, via the representation variety of $\Gamma$. 

\section{The direct finiteness conjecture and other outstanding problems}

Strong motivation for gathering such representation-theoretic data comes from the Direct Finiteness Conjecture. We briefly gather here what is known on this as well as connections to other outstanding problems. In \cite{K}, Kaplansky conjectured that if $G$ is an arbitrary group and $\KK$ is any field, whenever $xy = 1$ for $x,y \in \KK G$, then $yx = 1$ as well; that is, the group algebra is directly finite. In other words, the Direct Finiteness Conjecture is equivalent to the claim that $R$ does not embed in any group algebras. An algebra $A$ is said to be stably finite if all matrix algebras $M_n(A)$ are directly finite, and $G$ is stably finite when $\KK G$ is so. The direct finiteness conjecture for groups is known to be equivalent to the statement that every group is stably finite; also, $G$ is stably finite if $G\times H$ is directly finite for all $H$ (\cite{DJ}). 
 
The Direct Finiteness Conjecture of Kaplansky is known to hold in characteristic zero by \cite{M}, \cite{K}. It was first proved to also hold in characteristic $p>0$ for a wide class of groups - namely, for free-by-amenable groups - in \cite{AOP}. Later, in \cite{ES} it was shown that this conjecture is true in arbitrary characteristic for sofic groups - that is, groups that can be embedded into {\it metric} ultraproducts of finite groups. Sofic groups were originally introduced by Gromov \cite{G} motivated by an important problem in symbolic dynamics known as Gottschalk's surjunctivity conjecture (which he proved to hold for this class of groups in \cite{G}). This includes the result of \cite{AOP}. In fact, as it turns out, there are no known examples of non-sofic groups, and existence of sofic groups has become an important problem. More recently, in \cite{Be}, the conjecture is proved for \{finitely generated residually finite\}-by-sofic groups; this class includes some groups for which it is currently not known whether they are sofic or not. 
 
On the other hand, this connection of the direct finiteness conjecture has prompted interest from researchers coming from other directions (operator theory): \cite{Be, DHJ, DJ}. In \cite{DHJ} a class of finitely presented groups universal with respect to the direct finiteness conjecture is introduced, based on a computational idea: if $(\sum\limits_{i=1^n}\lambda_ig_i)(\sum\limits_{j=1}^k\mu_jh_j)=1$, then expanding and equating the two sides, one obtains a series of relations of the type $g_ih_j=g_kh_l$. This produces, for each pair $(n,k)$ a set of finitely many finitely presented groups, and the direct finiteness conjecture is equivalent to deciding direct finiteness for group algebras of these groups. Computational confirmation is obtained then for this conjecture for special small values of $(n,k)$ over the field with two elements $\FF_2$ and the associated groups \cite{DHJ}. 
 
Finally, it is worth mentioning the connection with another class of groups, known as hyperlinear groups. We refer to the the excellent surveys \cite{CLP,KP,Pv} for a history and state of the art of the subject, and basic definitions. Briefly, a hyperlinear group is a group that can be embedded into metric ultraproduct of the $n\times n$ unitary groups $U(n)$ over $\CC$. It is known that any sofic group is hyperlinear, but no examples of hyperlinear and non-sofic groups are known, and also, no example of non-hyperlinear groups is known. In fact, finding examples of groups that are not hyperlinear is equivalent to Connes' Embedding Conjecture for groups (by a result of Radulescu \cite{KP,Pv}), which is a central problem in mathematics and theoretical physics. This brings new light on direct finiteness: a counterexample would give the first examples of non-sofic groups, a proof could be considered further evidence for the Embedding Conjecture. 
 
In fact, in view of these and of the method of \cite{ES}, it is tempting to conjecture or at least ask the following question 

\begin{question}
Does every hyperlienar group satisfy the direct finiteness conjecture?
\end{question}

In \cite{ES}, the authors construct an embedding of a sofic group $G$ into an (algebraic) ultraproduct of finite matrix rings, which turns out to be a continuous von-Neumann regular ring, and so is endowed with a tracial norm $N$ with values in $[0,1]$ (pseudo-rank function). Then one can apply the classical proof of characteristic 0 of Kaplanski \cite{K} (von-Neumann algebras) or Montgomery \cite{M} ($C^*$-algebras): if $xy=1$, then $yx$ is an idempotent; then $1=N(xy)=N(yx)$ so $N(1-yx)=0$; but $N(e)=0$ for an idempotent implies $e=0$. This norm is constructed as a limit over the ultrafilter of the normalized usual rank functions on the finite matrix rings. The embedding of $G$ makes use of partial maps $f:F\rightarrow {\rm Map}(V_{F,\epsilon})$ for finite subsets  $F$ of $G$ and suitable finite sets $V_{F,\epsilon}$ (coming from the definition of sofic groups), which in turn are lifted to maps $G\rightarrow M_n(\KK)=\End(\KK V_{F,\epsilon})$. This construction would still be possible if the set of partial maps is $F\rightarrow U(n)$ for unitary groups (over $\CC$); the main hurdle here is that one would not have the liberty of introducing the coefficients in $\KK$, as the image of these partial maps are in $M_n(\CC)$. Perhaps one way around this is to consider the (finitely generated) subgroups generated by such images, and use Malcev's theorem (they must be residually finite) and try to use such finite pieces as in the maps above. However, at the same time, it may be that such an attempt would be close to proving that a hyperlinear group is sofic, which as mentioned, is an important problem in its own.

\subsection*{A possible strategy for direct finiteness}

Finally, we note a possible strategy approaching of the direct finiteness of group algebras, based on the properties of the Toeplitz algebra. If $\KK G$ is not directly finite, then there is $R\subset \KK G$ for $x,y\in G$ with $xy=1, yx\neq 1$; moreover, there are many such copies of $R$, generated by pairs of elements $(gxh^{-1},hyg^{1})$. In fact, as left (right) $R$-modules $\KK G=\sum\limits_g Rg$ $(=\sum\limits_ggR)$, and there is a filtration of $R$-modules $\Sigma\subseteq F\subseteq \KK[G]$ where $\Sigma$ is the $S$-socle, $F/\Sigma$ is locally finite and is the $\KK[x,x^{-1}]$-torsion part of $\KK[G]/\Sigma$, and $\KK[G]/F$ is torsion free (over $R$ and $\KK[x,x^{-1}]$). Furthermore, this filtration is one of right ideals in $\KK[G]$, and then $G$ acts as automorphisms on $\Sigma$, $F/\Sigma$, $\KK[G]/F$ (by right multiplication). In view of the above technique involving splitting pairs, a possibly fruitful source of information could be looking at the left-$\KK[x]$ (and/or right $\KK[y]$-module) structure of $\KK[G]$, as there will be a large injective direct summand (coproduct of injective hulls of $\KK[x]/(x)$). Finally, not discussed here, the $R$-bimodule structure of $\KK[G]$ may yield new information, and so $R$-bimodules \cite{B} would become relevant as well.

\vspace*{3mm} 
\begin{flushright}
\begin{minipage}{148mm}\sc\footnotesize

Miodrag Cristian Iovanov\\
University of Iowa, 14 MacLean Hall, Iowa City, IA 52242\&\\
University of Bucharest, Faculty of Mathematics, Str.Academiei 14,
RO-70109, Bucharest, Romania \\
{\it E--mail address}: {\tt miodrag-iovanov@uiowa.edu
}\vspace*{3mm}

Alexander Sistko\\
University of Iowa, 14 MacLean Hall, Iowa City, IA 52242\\
{\it E--mail address}: {\tt alexander-sistko@uiowa.edu
}\vspace*{3mm}

\end{minipage}
\end{flushright}


\begin{thebibliography}{J\c{S}}




\bibitem{Ab1} 
G. Abrams, G. \emph{Leavitt Path Algebras: The First Decade}, (2014) arXiv/math:1410.1835; http://arxiv.org/abs/1410.1835. 

\bibitem{Ab2} 
G. Abrams,  F. Mantese, A. Tonolo, \emph{Extensions of Simples Modules over Leavitt Path Algebras}, (2015) arXiv/math:1408.3580 19 January 2015. 

\bibitem{AAJZ} 
A. Alahmedi, H. Alsulami, S. Jain, E. Zelmaov, \emph{Structure of Leavitt Path Algebras of Polynomial Growth},  Proc. Natl. Acad. Sci. USA 110 (2013), no. 38, 15222--15224. 

\bibitem{AOP}
P. Ara, K.C. O�Meara, F. Perera, \emph{Stable finiteness of group rings in arbitrary characteristic}, Adv. Math. 170 (2002), 224-�238.

\bibitem{AMP} 
P. Ara, M.A. Moreno, E. Pardo, \emph{Non-Stable K-Theory for graph algebras}, Algebr. Represent. Theory 10 (2007), No.2, 157--178. 

\bibitem{AR} 
P. Ara, M. Rangaswamy, \emph{Finitely Presented Simple Modules over Leavitt Path Algebras}, J. Algebra 417 (2014), 333--352. 

\bibitem{AB} 
P. Ara, M. Brustenga, \emph{Module Theory over Leavitt Path Algebras and K-Theory}, J. Pure Appl. Algebra 214 (2010), No. 7, 1131--1151.  

\bibitem{B} 
V. Bavula, \emph{The Algebra of One-Sided Inverses of a Polynomial Algebra}, J. Pure Appl. Algebra 214, No. 10 (2010), 1874--1897.

\bibitem{Be}
F. Berlai, \emph{Groups satisfying Kaplanski's stable finiteness conjecture}, (2015) arXiv/math:1501.02893v1.

\bibitem{CLP}
V. Caprano, M. Lupini, V. Pestov, \emph{Introduction to sofic and hyperlinear groups and Connes' embedding conjecture}, (2013), arXiv/math:1309.2034v4.

\bibitem{Co} 
P. Colak, \emph{Two-Sided Ideals in Leavitt Path Algebras}, J. Algebra Appl. 10 (2011), No. 5, 801. 

\bibitem{Ch} 
X-W. Chen, \emph{Irreducible Representations of Leavitt Path Algebras}, Forum Math. 27 (2012), No. 1, 549�-574.


\bibitem{DHJ}
K. Dykema, T. Heister, K. Juschenko, \emph{Finitely presented groups related to Kaplansky's direct finiteness conjecture}, Exp. Math. 24 (2015), No. 3, 326�-338. arXiv/math:1112.1790

\bibitem{DJ}
K. Dykema, K. Juschenko, \emph{On stable finiteness of group rings}, Algebra Discrete Math. 19 (2015), No. 1, 44�-47. 

\bibitem{ES} 
G. Elek, E. Szab\'o, \emph{Sofic Groups and Direct Finiteness}, J. Algebra 280 (2004), No. 2, 426�-434. 

\bibitem{G}
M. Gromov, \emph{Endomorphisms of symbolic algebraic varieties}, J. Eur. Math. Soc. 1 (1999), 109�-197.

\bibitem{K} 
I. Kaplansky, \emph{Fields and Rings}, University of Chicago Press, Chicago, IL, 1969.   

\bibitem{KP}
A. Kwiatkowska, V. Pestov, \emph{An introduction to hyperlinear and sofic groups}, preprint, 2014.

\bibitem{M} 
S. Montgomery, \emph{Left and Right Inverses in Group Algebras}, Bull. Amer. Math. Soc. 75 (1969), No. 3, 539--540. 

\bibitem{M1}
S. Montgomery, \emph{von Neumann finiteness of tensor products of algebras}, Comm. Algebra 11 (1983), No. 6, 595�-610.

\bibitem{pass1}
D. Passman, Infinite crossed products. Pure and Applied Mathematics, 135. Academic Press, Inc., Boston, MA,
1989.

\bibitem{pass2}
D. S. Passman, \emph{The Algebraic Structure of Group Rings}, Wiley-Interscience, New York, 1977. 

\bibitem{Pv} 
V. G. Pestov, Hyperlinear and sofic groups: a brief guide. \emph{Bull. Symbolic Logic} 14 (2008), No. 4, 449�-480.









\end{thebibliography}
\end{document}